\documentclass[12pt]{article}

\usepackage{latexsym,amstext,amsmath,amssymb,amsthm, amsopn}

\usepackage{color}
\definecolor{a}{rgb}{0.1,0.1,0.9} 
\definecolor{r}{rgb}{0.9,0.1,0.1} 

\newtheorem{thm}{Theorem}

\newtheorem{lem}{Lemma}
\newtheorem{cor}[lem]{Corollary}

\def \EE {\mathbb{E}}
\def \PP {\mathbb{P}}
\def \RR {\mathbb{R}}

\def \Rd {{\RR^d}}

\title{Time-dependent Schr\"odinger perturbations of transition
  densities\footnote{2000 MS Classification: 47A55, 60J35, 60J57. Keywords: relative
    Kato condition, conditional process.}}
\author{Krzysztof Bogdan\footnote{This author was partially supported by KBN.}, Wolfhard Hansen, Tomasz Jakubowski\footnotemark[2] 
}

\date{\today}
\begin{document}
\baselineskip=17pt
\maketitle
\begin{abstract}
We construct the fundamental solution
of $\partial_t-\Delta_y- q(t,y)$, for functions $q$ with a
certain integral space-time relative smallness, in particular for
those satisfying a relative Kato condition. 
The resulting transition density is comparable to the Gaussian kernel in finite
time, and it is even asymptotically equal 
to the Gaussian kernel (in small time) 
under the relative Kato condition.

The result is generalized to arbitrary strictly positive and finite time-nonhomogeneous 
transition densities on measure spaces.

We also discuss specific applications to Schr\"odinger perturbations of the
fractional Laplacian in view of the fact that the 3P Theorem holds for the
fundamental solution corresponding to  the operator.
\end{abstract}

\section{Main results and overview}
\label{sec:intro}
Let $d$ be a natural number. The Gaussian kernel on $\Rd$ is defined as
$$
g(s,x,t,y)=
\frac{1}{\big(4\pi(t-s)\big)^{d/2}}\exp\left(-\frac{|y-x|^2}{4(t-s)}\right)\,,\mbox{
  if } -\infty<s<t<\infty\,, 
$$ 
and we let $g(s,x,t,y)=0$ if $s\geq t$. Here $x,y\in \Rd$ are
arbitrary. It is well-known that $g$ is a time-homogeneous 
transition density with respect to the Lebesgue measure, $dz$, on
$\Rd$. In particular, for $x,y\in \Rd$,
$$
\int_\Rd g(s,x,u,z)g(u,z,t,y)\,dz=g(s,x,t,y)\,,\quad\mbox{ if } s<u<t\,.
$$
We will consider a Borel measurable function $q\,:\;\RR\times \Rd \to
\RR$, and numbers $h>0$ and $0 \leq \eta <1$,
such that for all $x,y\in \Rd$ and $s < t\leq s+h$, 
\begin{equation}\label{con:scqp}
\int_s^t \int_{\Rd}
 \frac{g(s,x,u,z)g(u,z,t,y)}{g(s,x,t,y)}|q(u,z)| \,dz\,du
\le \eta \,.
\end{equation}
\begin{thm}\label{Theorem1}
There is a unique continuous transition density $\tilde{g}$ such that 
\begin{equation}\label{eq:fs}
\int\limits_{s}^\infty \int\limits_{\Rd}
\tilde{g}(s,x,t,y)
\big[\partial_t\phi(t,y)+\Delta_y \phi(t,y)+q(t,y)\phi(t,y)\big]\,dtdy = -\phi(s,x)\,,
\end{equation}
$$
\mbox{ and }\quad 
\frac{\tilde{g}(s,x,t,y)}{g(s,x,t,y)} 
\leq 
\frac{1}{1-\eta} \exp\big\{\frac{\eta }{(1-\eta)h}(t-s)\big\}
\,.
$$
Here $s<t\in \RR$, $x,y\in \Rd$, and $\phi\in C^\infty_c(\RR\times
\Rd)$ are arbitrary. 
\end{thm}
We consider $\Delta+q$ as an additive perturbation of the Laplacian
$\Delta$ by the operator of multiplication by $q$ (i.e. a
Schr\"odinger perturbation).
According to (\ref{eq:fs}), $\tilde{g}$ is 
the integral kernel of the left inverse of $-\big(\partial_t+\Delta_y+q\big)$.
Put differently, $f:(t,y)\mapsto \tilde{g}(s,x,t,y)$ solves 
$(\partial_t-\Delta_y-q) f=\varepsilon_{(s,x)}$ in the sense
of distributions. Thus, $\tilde{g}$ is the fundamental solution of $\partial_t-\Delta_y-q$ 
(\cite{MR1978999}).

As we will see, $\tilde{g}$ is constructed by means of $g$ and $q$ only, without
referring to~$\Delta$. A similar procedure applies to
the fundamental solution of the fractional Laplacian
$\Delta^{\alpha/2}=-(-\Delta)^{\alpha/2}$. At the end of the paper we give references
and discuss 
these two important examples in some detail.

The primary goal of the paper is, however, to construct and estimate analogous 
time-dependent, or non-autonomous, Schr\"odinger perturbations for 
more general transition densities. 
We work under the appropriate
assumption of {\it relative smallness}, or {\it relative Kato} condition on $q$.
We give explicit upper bounds for the resulting transition density, which are new even
in the autonomous Gaussian case. 
Our development is motivated by the role of the celebrated
3G Theorem in studying Schr\"odinger-type perturbations of Green functions
\cite{MR2160104,MR2207878,MR2195185}, see also
\cite{MR936811}, \cite{MR1329992}, \cite{MR1473631}, \cite{MR1671973}, \cite{MR1825645},
\cite{MR1920109}.
Another motivation comes from  a recent estimate, the 3P Theorem of
\cite{MR2283957} for the fundamental solution of $\Delta^{\alpha/2}$.
The estimate was used in \cite{MR2283957} to construct the transition
density of autonomous {\it gradient} perturbations of $\Delta^{\alpha/2}$,
in a way resembling the above mentioned study of Schr\"odinger
perturbations of Green functions by means of the 3G Theorem (see also \cite{J3}).
\cite{MR2283957} and the present paper show that a perturbation
technique similar to that of \cite{MR2207878}
applies even more naturally to the {\it parabolic} Green
function (that is,  the fundamental solution, or transition density). 
We propose an explicit construction of
the perturbed transition density under a minimum of assumptions, corresponding
with the generality of \cite{MR2160104,MR2207878,MR2195185,MR1920109}.
We refer the reader to \cite[Theorem 3.5, p. 71]{MR591851},
\cite{MR1335452}, and \cite[Theorem 19, Ch.VIII]{MR0117523} for a
selection of results in the perturbation theory of lineal operators
and semigroups on Banach spaces. 
For recent developments in Schr\"odinger perturbations of
time-nonhomogeneous transition probabilities
we refer the reader to \cite{MR2253111},
\cite{MR2253015}, \cite{MR2395164}.
We like to point out that our main estimate, Theorem~\ref{th:optyl}
below, is more precise and explicit than those mentioned above. 
Accordingly, it also strengthens in the given context the
celebrated Khasminski's lemma, one of the main tools in the probabilistic theory
of Schr\"odinger perturbations of generators of Markov processes (see, e.g.,
\cite{MR1329992}, \cite{MR2395164}).
This strengthening is of independent interest --- the estimate 
is valid in the full range of times, rather than only in small time
intervals, and the proof gives a deeper insight into the interplay
between individual terms of the series involved.

To explain the connection of Theorem~\ref{Theorem1} to our general
results we note that $\tilde{g}$ satisfies  the following 
 equation for all $x,y\in \Rd$ and $s,t\in \RR$,
$$
\tilde{g}(s,x,t,y)=g(s,x,t,y)+\int_\RR\int_X
g(s,x,u,z)q(u,z)\tilde{g}(u,z,t,y)\,dzdu
$$
(see the proof of Theorem~\ref{Theorem1g}).
This is called \emph{Duhamel's formula} or \emph{perturbation
  formula}. The equation implicitly defines the perturbed transition 
density in this, and in a more general situation, which we will now discuss.

Consider an arbitrary set $X$ with a $\sigma$-algebra $\mathcal{M}$ and
a (nonnegative)
measure $m$ defined on $\mathcal{M}$. To simplify the notation we will
write $dz$ for $m(dz)$ in what follows. Consider the $\sigma$-algebra
$\mathcal{B}$ of Borel subsets of $\RR$, and the Lebesgue measure, $du$,
defined on $\mathcal{B}$. The {\it space-time},
$\RR\times X$, will be equipped with 
$\sigma$-algebra $\mathcal{B}\otimes\mathcal{M}$ and  product
measure $du\,dz=du\,m(dz)$. 

Let $p$ be a 
$\mathcal{B}\times \mathcal{M}\times\mathcal{B}\times\mathcal{M}$-measurable 
function defined (everywhere) on $\RR\times X\times \RR\times X$.
We will call $p$ a {\it transition density} on $X$ if
\begin{equation}
  \label{eq:conzero}
p(s,x,t,y)=0, \quad \mbox{ whenever } s\geq t\,,   
\end{equation}
\begin{equation}
  \label{eq:nn}
0<p(s,x,t,y)<\infty, \quad \mbox{ when } s< t\,,\; x,y\in X\,,  
\end{equation}
and the following Chapman-Kolmogorov equation holds, if $s<u<t$,
\begin{equation}
  \label{eq:ChK}
  \int_X p(s,x,u,z)p(u,z,t,y)\,dz=p(s,x,t,y)\,,\quad x,y\in X\,.
\end{equation}
Note that we require strict positivity in (\ref{eq:nn}), while
(\ref{eq:conzero}) is merely a convention. By (\ref{eq:nn}) and (\ref{eq:ChK}),
$m$ is necessarily $\sigma$-finite on $X$.
The reader may regard $s$ and $x$ in $p(s,x,t,y)$
as the starting time and position of a {\it variable} mass spreading
in $X$, and  $t$, $y$ as the ending time and position. Thus,
$\int_X p(s,x,t,y)dy$ is the total mass at time
$t$.
We will say that transition densities $p'$ and $p''$ on $X$ are
{\it comparable locally in time} if for every $h>0$ there is a (finite)
constant $c=c(h)$ such that $c^{-1}p''(s,x,t,y)\leq p'(s,x,t,y)\leq c p''(s,x,t,y)$ 
for all $x,y\in X$ provided $s<t<s+h$. We will say that they are {\it
  asymptotically equal}  if
$c(h)$ may be  chosen in such a way 
that $c(h)\to 1$ as $h\to 0^+$.

All the functions discussed in
the sequel will be assumed 
measurable with respect to the relevant
$\sigma$-algebras, usually with respect to $\mathcal{B}\times\mathcal{M}$. 

If $q$ is a {\it nonnegative} function defined on
$\mathcal{B}\times\mathcal{M}$, then we let
\begin{equation}
  \label{eq:des}
\eta^*(q)=\inf \eta\,,
\end{equation}
where the infimum is taken over all $\eta>0$ with the property that 
there exists
$h>0$ such that for all $x,y\in \Rd$ and $s < t\leq s+h$,
\begin{equation}\label{con:scp}
\int_s^t 
\int_{X}
p(s,x,u,z)q(u,z)p(u,z,t,y)
\,dzdu
\le \eta \,p(s,x,t,y)\,.
\end{equation}
We will say that a (signed) function $q:\,\RR\times X\to \RR$ is
{\it relatively bounded} (at small times, with respect to $p$ and
$m$) if $0\leq \eta^*(|q|)<\infty$. We will say that $q$
is {\it relatively small} if $0\leq \eta^*(|q|)<1$, and we will say that $q$
is {\it relatively Kato}\footnote{A {\it different} ``relative Kato condition''
  is used in \cite[(4)]{MR2137058} and
  \cite[(2)]{MR2320691}.} 
if $\eta^*(|q|)=0$.


\begin{thm}\label{Theorem1g}
Consider a real-valued function $q$ on $\RR\times X$.
If $q$ is relatively small, then there is a unique
transition density $\tilde{p}$ on $X$ locally in time comparable with $p$, such that 
for all $s,t\in \RR$, and $x,y\in  X$, we have 
\begin{equation}\label{eq:df}
\tilde{p}(s,x,t,y)=p(s,x,t,y)+\int_\RR\int_X
p(s,x,u,z)q(u,z)\tilde{p}(u,z,t,y)\,dzdu\,.
\end{equation}
If $q$ is relatively Kato, then $\tilde{p}$ and $p$ are
asymptotically equal.
\end{thm}
We note that explicit upper and lower bounds for
$\tilde{p}$ exist expressed in terms of $\eta^*(q_+)$ and $\eta^*(q_-)$,
see Theorem~\ref{th:optyl} and (\ref{eq:comp}) (see also (\ref{eq:optdK})).

The paper is organized as follows. In Section~\ref{sec:al} we describe
the basic formalism of the {\it perturbation series} in the context of the
Chapman-Kolmogorov equation.
In Section~\ref{sec:san} we reformulate the relative
  boundedness and smallness of $q$. By a combinatorial argument, 
we prove our main estimate, Theorem~\ref{th:optyl}, 
for the perturbation series for relatively small $q\geq 0$.
In Section~\ref{sec:gp} we consider signed
relatively small $q$ and we prove Theorem~\ref{Theorem1g}.
In Section~\ref{sec:ss} we discuss in more detail Schr\"odinger
perturbations of the transition densities of Laplacian and fractional Laplacian, 
and we give the proof of Theorem~\ref{Theorem1}.
In view of the fact that the transition density of the fractional
Laplacian (but not that of the Laplacian) satisfies the 3P Theorem,
in Section~\ref{sec:ss} we characterize
relative Kato condition by means of the parabolic Kato condition
studied in \cite{MR1488344}. 

Our main goal is to give applications motivating the use of 
{\it relative smallness} in perturbation theory of
transition densities, along with a self-contained exposition of
some of the relevant techniques. We do not attempt full generality here.
Possible and forthcoming generalizations are mentioned 
in Section~\ref{sec:fd}, where we also give a probabilistic
interpretation of our results.

\section{Algebra of perturbation series}
\label{sec:al}
Let $q:\,\RR\times X\to \RR$. 
The identities we intend to prove below rely merely on changing the order of
integration, which is justified if the integrals involved are
absolutely convergent or {\it nonnegative}.  
We shall first consider the latter situation and we will {\it assume that $q\ge 0$}.

Duhamel's formula (\ref{eq:df}) suggests the following definitions.
For $s,t\in \RR$ and $x,y \in X$,  we let $p_0(s,x,t,y)  =  p(s,x,t,y)$,  
\begin{equation}
\label{eq:18.75}
 p_n(s,x,t,y)  =  
\int_s^t\int_{X} 
p_{n-1}(s,x,u,z)p(u,z,t,y)\,q(u,z)\,dzdu\,,
\end{equation}
for $n \ge 1$, and we define the perturbation of $p$ by $q$,
\begin{equation}
  \label{eq:defk}
  \tilde{p}_q(s,x,t,y)=\sum_{n=0}^\infty p_n(s,x,t,y)\,,\quad x,y\in  X\,,\quad
  s,t\in \RR\,.
\end{equation}
If $s\ge t$, then $p_n(s,x,t,y) = 0$ for every $n\ge 0$ and hence
$\tilde{p}_q(s,x,t,y)=0$.

Since $p(s,x,t,y)=0$ for $s\geq t$, we could write (\ref{eq:18.75}) as 
$$
 p_n(s,x,t,y)  =  
\int_\RR\int_{X} 
p_{n-1}(s,x,u,z)p(u,z,t,y)\,q(u,z)\,dzdu\,,
$$
for {\it all} $s,t\in \RR$ and $x,y \in X$, so the reader should not
be alarmed if we occasionally simplify our notation in this way. 
\begin{lem}\label{pmnm}
For  all $s<u<t$, $x,y\in X$, and $n=0,1,\ldots$, 
\begin{equation}\label{eq:lem}
\sum_{m=0}^n  \int_{ X} p_m(s,x,u,z)  p_{n-m}(u,z,t,y)\,dz = p_n(s,x,t,y)\,.
\end{equation}
\end{lem}
\begin{proof}
We note that (\ref{eq:lem}) is true for $n=0$ by (\ref{eq:ChK}). Assume that $n\geq 1$
and (\ref{eq:lem}) holds for $n-1$.  The sum of the 
 first $n$ terms in
(\ref{eq:lem}) can be dealt with by induction: 
\begin{eqnarray}
&& 
\sum_{m=0}^{n-1} \int_{ X}  p_m(s,x,u,z)  p_{n-m}(u,z,t,y)\,dz \nonumber\\
& = &   
\sum_{m=0}^{n-1} \int_{ X}   p_m(s,x,u,z)  
\int_u^t \int_ X  p_{n-1-m}(u,z,r,w) p(r,w,t,y) \,q(r,w)dwdr \,dz \nonumber\\
& = &   \int_u^t \int_{ X} \left(\sum_{m=0}^{n-1}  \int_{ X} 
p_m(s,x,u,z) p_{n-1-m}(u,z,r,w)\,dz\right) p(r,w,t,y) \,q(r,w)dwdr\nonumber\\
& = &   \int_u^t \int_{ X}  p_{n-1}(s,x,r,w) p(r,w,t,y) \,q(r,w)dwdr\,. \label{eq1:lem}
\end{eqnarray}
By (\ref{eq:18.75}), the $(n+1)$-st term is 
\begin{eqnarray}
&& \int_{ X} p_n(s,x,u,z)  p_0(u,z,t,y) \,dz\nonumber\\
& = &   \int_{ X} 
\int_s^u \int_{ X}  
p_{n-1}(s,x,r,w)p(r,w,u,z) \,q(r,w)dwdr\;  p(u,z,t,y)\,dz \nonumber\\
& = &   \int_s^{u} \int_{ X}  p_{n-1}(s,x,r,w)  p(r,w,t,y) \,q(r,w)dwdr\,. \label{eq2:lem}
\end{eqnarray}
and (\ref{eq:lem}) follows adding (\ref{eq1:lem})  and (\ref{eq2:lem}).
 \end{proof}

We next prove the Chapman-Kolmogorov equation for $\tilde{p}_q=\sum_{n=0}^\infty p_n$. 

\begin{lem}\label{t:chapkol}
For all $s<u<t$ and $x,y \in  X$, 
$$
\int_{ X} \tilde{p}_q(s,x,u,z)  \tilde{p}_q(u,z,t,y) \,dz = \tilde{p}_q(s,x,t,y)\,.
$$
\end{lem}
\begin{proof}
By 
Lemma \ref{pmnm}, 
\begin{eqnarray*}
\int_{ X} \tilde{p}_q(s,x,u,z)  \tilde{p}_q(u,z,t,y) \,dz & =&  
\int_{ X} \sum_{i=0}^\infty p_i(s,x,u,z) \sum_{j=0}^\infty p_j(u,z,t,y) \,dz\\
& = &    \sum_{n=0}^\infty    \sum_{m=0}^n  \int_{ X}
p_m(s,x,u,z)  p_{n-m}(u,z,t,y) \,dz \\
& = & \sum_{n=0}^\infty p_n(s,x,t,y) = \tilde{p}_q(s,x,t,y)\,.
\end{eqnarray*}
\end{proof}
We will need the following extension of (\ref{eq:18.75}). 
\begin{lem}\label{lem:18.76}
For all $n=1,2,\ldots$, $m=0,1,\ldots,n-1$, $s,t\in \RR$ and $x,y \in  X$, 
\begin{equation}
\label{eq:18.76}
 p_n(s,x,t,y)  =  
\int_s^t\int_{ X} 
p_{n-1-m}(s,x,u,z) p_{m}(u,z,t,y)\,q(u,z)\,dzdu\,.
\end{equation}
\end{lem}

\begin{proof}
For $m=0$, equality (\ref{eq:18.76}) holds by  definition of $p_n$. In
particular, this proves our claim for $n=1$. If $n\geq 1$
such that (\ref{eq:18.76}) holds, then, for every 
$m=1,2,\ldots,n$,
\begin{eqnarray*}
&&  p_{n+1}(s,x,t,y)=\int_\RR \int_ X 
p_{n}(s,x,u,z) p(u,z,t,y)\,q(u,z)\,dzdu\\
&=&
\int_\RR \int_ X \int _\RR \int_ X  
p_{n-1-(m-1)}(s,x,w,v)p_{m-1}(w,v,u,z)\,q(v,w)dwdv\\
&&\qquad\qquad\qquad\qquad \qquad\qquad\qquad \qquad\qquad
p(u,z,t,y)\,q(u,z)\,dzdu\\
&=&
\int _\RR \int_ X  
p_{n-m}(s,x,w,v)p_{m}(w,v,t,y)\,q(v,w)dwdv\,.
\end{eqnarray*}
\end{proof}
\section{Estimate from above}
\label{sec:san}
In this section we will only consider {\it relatively bounded $q\geq 0$}.
Given $s<t$, we let $I(s,t)$ be the smallest number
such that for all $x,y \in  X$, 
\begin{equation}
  \label{eq:dI}
\int_s^t \int_{ X}
 p(s,x,u,z)p(u,z,t,y) \,q(u,z)\,dzdu
\le I(s,t) \,p(s,x,t,y)\,.
\end{equation}
Relative boundedness of $q$ implies that $I(s,t)$ is finite, if
$t-s$ is small. The following lemma yields that then 
$I(s,t)$ is  finite for \emph{all} $s<t$.
\begin{lem}\label{lem:sa}
$I(s,v)\leq I(s,t)+I(t,v)$, whenever $s<t<v$. 
\end{lem}
\begin{proof}
Let $s<t<v$ and $x,y\in  X$. We have
\begin{eqnarray*}
&& \int_s^v \int_{ X}
p(s,x,u,z)p(u,z,v,y)\,q(u,z)\,dzdu
= \int_s^t+\int_t^v\\
&=&
\int_s^t\int_{ X}\int_ X p(s,x,u,z)p(u,z,t,w)p(t,w,v,y)dw\,q(u,z)\,dzdu\\
&&+
\int_t^v\int_{ X}\int_ X p(s,x,t,w)p(t,w,u,z)p(u,z,v,y) dw\,q(u,z)\,dzdu\\
&\leq&
\left[I(s,t)
+
I(t,v)\right]\int_{ X} p(s,x,t,w)p(t,w,v,y)dw\\ 
&=&\left[I(s,t)+I(t,v)\right]p(s,x,v,y)\,.  
\end{eqnarray*}
\end{proof}

This {\it subadditivity} and the relative boundedness yield the
following. 
\begin{lem}\label{lem:intercept}
If $\eta>\eta^*(q)$, then there is $\beta\geq 0$ such
that 
\begin{equation}\label{con:A1}
\int_s^t \int_{ X}
 p(s,x,u,z)p(u,z,t,y) \,q(u,z)\,dzdu
\le [\eta + \beta(t-s)]\,p(s,x,t,y)\,,
\end{equation} 
whenever $s<t\in \RR$ and $x,y\in  X$. 
\end{lem}
\begin{proof}
 For $\eta>\eta^*(q)$ let $h>0$ be such as required in the paragraph between 
the definition (\ref{eq:des}) and the inequality (\ref{con:scp}). 
If  $k$ is a natural number and 
$s+(k-1)h<t\leq s+kh$, 
then $k<1+(t-s)/h$, and, by Lemma~\ref{lem:sa},
$I(s,t)\leq k\eta\leq \eta+\eta(t-s)/h$. We can take 
\begin{equation}
  \label{eq:ffb}
\beta=\eta/h\,.  
\end{equation}
\end{proof}
Conversely, if (\ref{con:A1}) holds with some finite $\eta$ and
$\beta$, then $q$ is relatively bounded, and $\eta^*(q)\leq \eta$.
Also, (\ref{con:A1}) with $0\leq \eta<1$ (and some finite $\beta$) characterizes relative
smallness, and (\ref{con:A1}) being true for {\it every} $\eta\geq 0$ (with some
finite $\beta$) is equivalent to the 
 relative Kato condition.
Thus, our focus in (\ref{con:A1}) is on the value of $\eta$; the term $\beta(t-s)$ is merely a technically convenient
replacement of $h$.

In the remainder of the section we will assume that (\ref{con:A1}) 
holds with $0\leq \eta<1$ and (finite) $\beta\geq 0$. 
For instance, every bounded (nonnegative) $q$ satisfies the assumption with
$\eta=0$ and $\beta=\sup_{\RR\times X} q(u,z)$.  Indeed,
$$
\int_s^t \int_{ X}
 p(s,x,u,z)p(u,z,t,y) \,q(u,z)\,dzdu
\le 
[\sup q]
(t-s)p(s,x,t,y)\,.
$$

We shall need the following identity. 

\begin{lem}\label{lem:bc}
For all $n=0,1,\ldots$ and $\xi,\nu,\eta \in \RR$, 
\begin{equation}\label{eq:3sum}
\sum_{m=0}^{n} \sum_{k=0}^{n-m}\sum_{j=0}^m 
\binom{n-m}{k}\binom{m}{j} \frac{\xi^k \nu^j}{k!j!}\eta^{n-k-j} 
= \sum_{r=0}^n \binom{n+1}{r+1} \frac{(\xi + \nu)^r}{r!} \eta^{n-r}\,.
\end{equation} 
\end{lem}

\begin{proof} 
Let $j,r$ be integers, $0\le j\le r\le n$. There are $\binom{n+1}{r+1}$
subsets of $\{0,1,\dots,n\}$ having $r+1$ elements. Such a subset, 
$\{i_1,i_2,\dots,i_{r+1}\}$ with $i_1<i_2<\dots <i_{r+1}$, may be
chosen by first fixing $m:=i_{j+1}\in\{j,j+1,\dots,n-r+j\}$ and then
taking $j$ integers $0\le i_1<\dots<i_j<m$, and $r-j$ integers
$m<i_{j+2}<\dots<i_{r+1}\le n$. This shows that (for every such $j$)
$$
\sum_{j\leq m\leq n-r+j} 
\binom{m}{j} \binom{n-m}{r-j}=\binom{n+1}{r+1}\,
$$
(compare \cite[(5.26)]{MR1397498}).
Considering $k = r-j$ we see that
the conditions $0\leq r\leq n$, $0\leq j\leq r$, $j\leq m \leq
n-r+j$ are equivalent to $0\leq m\leq n$, $0\leq k \leq n-m$, 
$0\leq j\leq m$. Therefore
\begin{eqnarray*}
&&\sum_{m=0}^{n} \sum_{k=0}^{n-m}\sum_{j=0}^m 
\binom{n-m}{k}\binom{m}{j} \frac{\xi^k \nu^j}{k!j!}\eta^{n-k-j} \\
&=&\sum_{r=0}^{n} \sum_{j=0}^{r}\sum_{m=j}^{n+j-r}
\binom{m}{j} \binom{n-m}{r-j}
\frac{\nu^j \xi^{r-j}}{j!(r-j)!}\eta^{n-r}\\
&=&\sum_{r=0}^{n} \binom{n+1}{r+1} 
\frac{(\nu+\xi)^{r}}{r!}\eta^{n-r}\,.
\end{eqnarray*}
\end{proof}
The following result is our main technical observation,
validating the worth of the assumption (\ref{con:A1}).

\begin{lem}\label{technical}
For all $n=0,1,2,\dots$, $x,y \in  X$ and $s,t \in \RR$, 
\begin{equation}
  \label{eq:oopn}
p_n(s,x,t,y)
\leq p(s,x,t,y)\sum_{k=0}^n \binom{n}{k} \frac{(\beta(t-s))^k}{k!}\eta^{n-k}\,.
\end{equation}
\end{lem}

\begin{proof}
Of course, (\ref{eq:oopn}) holds for $n=0$.
By Lemma~\ref{lem:18.76},
induction, Lemma~\ref{lem:bc}, and (\ref{con:A1}),
\begin{eqnarray*}
&&(n+1)p_{n+1}(s,x,t,y)  \\
&=& \sum_{m=0}^{n}  \int_s^t\int_{ X}   p_{n-m}(s,x,u,z)p_{m}(u,z,t,y)\,q(u,z)\,dzdu \\
& \le & \int_s^t \int_{ X}  p(s,x,u,z)p(u,z,t,y)q(u,z)\,dzdu \\
&&\sum_{m=0}^{n} \sum_{k=0}^{n-m} \sum_{j=0}^m 
\binom{n-m}{k} \binom{m}{j}
\frac{(\beta(u-s))^k}{k!}\eta^{n-m-k}  
\frac{(\beta(t-u))^j}{j!}\eta^{m-j} \\
& = & \int_s^t \int_{ X}  p(s,x,u,z)p(u,z,t,y)\,q(u,z)\,dzdu 
\sum_{r=0}^n \binom{n+1}{r+1} \frac{(\beta(t-s))^r}{r!} \eta^{n-r}\\
& \le & p(s,x,t,y)[\eta + \beta(t-s)]\sum_{r=0}^n \binom{n+1}{r+1} \frac{(\beta(t-s))^r}{r!}\eta^{n-r}\\
& = & p(s,x,t,y)
\left[ \sum_{r=0}^n \binom{n+1}{r}
\frac{n+1-r}{r+1} \frac{(\beta(t-s))^r}{r!}\eta^{n+1-r}\right.\\
&&  \qquad\qquad\qquad\qquad\qquad\qquad\qquad
+ \left.\sum_{r=1}^{n+1}
 \binom{n+1}{r}r \frac{(\beta(t-s))^r}{r!}\eta^{n+1-r}\right]\\
&\le& (n+1) p(s,x,t,y)\sum_{r=0}^{n+1} \binom{n+1}{r} \frac{(\beta(t-s))^r}{r!}\eta^{n+1-r}\,.
\end{eqnarray*} 
\end{proof}

\begin{thm}\label{th:optyl}
If $q$ satisfies {\rm(\ref{con:A1})} with $\eta<1$, then,
 for all $s<t$ and $x,y\in  X$,
  \begin{equation}
    \label{eq:optyl}
\tilde{p}_q(s,x,t,y) \leq \frac{1}{1-\eta}
\exp\left(\frac{\beta}{1-\eta}(t-s)\right)\,
p(s,x,t,y)\,.
  \end{equation}
\end{thm}

\begin{proof}
By (\ref{eq:defk}) and Lemma \ref{technical},
\begin{eqnarray*}
\tilde{p}_q(s,x,t,y)
&\leq & p(s,x,t,y)
\sum_{k=0}^\infty \frac{(\beta(t-s))^k}{k!}
\sum_{n=k}^\infty \binom{n}{k}\eta^{n-k}\,,
\end{eqnarray*}
where
$$
\sum_{n=k}^\infty \binom{n}{k}\eta^{n-k}=
\frac{1}{k!}\frac{d^k}{d\eta^k}\sum_{n=0}^\infty \eta^{n}=
\frac{1}{k!}\frac{d^k}{d\eta^k}\frac{1}{1-\eta}
=\frac{1}{k!}\frac{k!}{(1-\eta)^{k+1}}
=\frac{1}{(1-\eta)^{k+1}}\,.
$$
Therefore
\begin{eqnarray}
\tilde{p}_q(s,x,t,y)
&\leq& p(s,x,t,y)
\frac{1}{1-\eta}
\sum_{k=0}^\infty \frac{(\beta(t-s)/(1-\eta))^k}{k!}\label{eq:lcp} \\
&=& 
\frac{1}{1-\eta}
\exp\left(\frac{\beta}{1-\eta}(t-s)\right)\,\nonumber
p(s,x,t,y)\,.
\end{eqnarray}
\end{proof}

\section{Small signed perturbations}
\label{sec:gp}
We will present some immediate consequences of
Theorem~\ref{th:optyl} for {\it signed} $q$. Let $q_+=\max(q,0)$,
$q_-=\max(-q,0)$, so that 
$
q=q_+-q_- 
$.  
It will be convenient to consider the following integral kernels 
on space-time $\RR\times X$
: 
$$
  Pf(s,x):=P_pf(s,x):=\int_\RR \int_ X  p(s,x,u,z)f(u,z)\,dzdu\,,
$$
\begin{equation}
  \label{eq:dpik}
  P^q f(s,x):=P_p^qf(s,x):=\int_\RR \int_ X  p(s,x,u,z)f(u,z)q(u,z)\,dzdu\,.
\end{equation}
For $t\in \RR$ and $y \in  X$ we have
\begin{eqnarray*}
\left(P^q p(\cdot,\cdot,t,y)\right)(s,x) & = & \int_\RR \int_{ X} p(s,x,u,z) p(u,z,t,y) \,q(u,z)\,dzdu \\
& = & p_1(s,x,t,y)\,,\quad s\in \RR\,,\;x\in  X\,.
\end{eqnarray*}
Using Lemma~\ref{lem:18.76} we obtain by induction that,
for every natural $n$, 
$$
\left((P^q)^n p(\cdot,\cdot,t,y)\right)(s,x)
=P^q(p_{n-1}(\cdot,\cdot,t,y))(s,x)
= p_n(s,x,t,y)\,.
$$
Relaxing notation, we can write 
$p_n(s,x,t,y)=(P^q)^n p(s,x,t,y)$, or even $p_n=(P^q)^np$.
In view of (\ref{eq:defk}) we define (for signed $q$)
\begin{equation}
  \label{eq:dpp}
  \tilde{p}_q=\sum_{n=0}^\infty (P^q)^np\,,
\end{equation}
whenever the integrals and the sum are nonnegative or absolutely
convergent. 
Given (signed) $q_1$ and $q_2$ we consider the effect of two consecutive perturbations,
$$
\widetilde{(\tilde{p}_{q_1})}_{q_2}:=\sum_{n=0}^\infty (P_{\tilde{p}_{q_1}}^{q_2})^n\tilde{p}_{q_1}\,.
$$
Here $P_{\tilde{p}_{q_1}}^{q_2}f(s,x)=\int_\RR \int_ X  \tilde{p}_{q_1}(s,x,u,z)f(u,z)q_2(u,z)\,dzdu$,
compare (\ref{eq:dpik}).
The following is a special case of a
general result on perturbations of integral kernels
(see, e.g., \cite[Problem 1.13]{MR758799}).
\begin{lem}\label{lem:aep}
Under the assumptions of nonnegativity or absolute convergence,  
\begin{equation}
    \label{eq:aep}
\widetilde{(\tilde{p}_{q_1})}_{q_2}=
\tilde{p}_{q_1+q_2}\,.      
  \end{equation}
\end{lem}
\begin{proof}
It is not hard to verify that 
$$
    P_{\tilde{p}_{q_1}}^{q_2}=\sum_{k=0}^\infty (P^{q_1})^k P^{q_2}\,.
$$
It follows that
\begin{eqnarray*} 
\widetilde{(\tilde{p}_{q_1})}_{q_2}&=&\sum_{n=0}^\infty (P_{\tilde{p}_{q_1}}^{q_2})^n \tilde{p}_{q_1}
=\sum_{n=0}^{\infty} 
\left(\sum_{k=0}^\infty(P^{q_1})^kP^{q_2}\right)^n
\sum_{l=0}^\infty (P^{q_1})^l p\\
&=&\sum_{n=0}^{\infty} \sum_{\sigma\in \{1,2\}^n}
P^{q_{\sigma(1)}}\ldots P^{q_{\sigma(n)}}p=\tilde{p}_{q_1+q_2}.
\end{eqnarray*} 
\end{proof}

\begin{proof}[Proof of Theorem~\ref{Theorem1g}]
We assume that $\eta^*(|q|)<1$, in particular $\eta^*(q_-)<1$ and $\eta^*(q_+)<1$.  
By Theorem~\ref{th:optyl}, $\tilde{p}_{q_-}=\sum_{n=0}^\infty (P^{q_-})^np$ is 
convergent, hence
$$\tilde{p}_{-q_-}=\sum_{n=0}^\infty (-1)^n(P^{q_-})^np$$ is well defined
and $|\tilde{p}_{-q_-}|\leq \tilde{p}_{q_-}$. 
Therefore the arguments of Section~\ref{sec:al} apply,
in particular $\tilde{p}_{-q_-}$ satisfies the Chapman-Kolmogorov equation, 
see Lemma~\ref{t:chapkol}. We will prove that $\tilde{p}_{-q_-}\geq 0$.
Indeed, let $\eta^*(q_-)<\eta<1$ and let $h>0$ be as required
between (\ref{eq:des}) and (\ref{con:scp}), that is, 
for all $s<t\leq s+h$,
$$
P^{q_-}p(s,x,t,y)\leq \eta p(s,x,t,y)\leq p(s,x,t,y)\,,\quad x,y\in X\,.
$$
Then
$(P^{q_-})^{n+1} p(s,x,t,y)\leq \eta (P^{q_-})^{n} p(s,x,t,y)$ for
$n=1,2,\ldots$, and hence
\begin{equation}
  \label{eq:lbbp}
\tilde{p}_{-q_-}
=\left(p-P^{q_-}p\right)+\left((P^{q_-})^2 p-(P^{q_-})^3 p\right)+\ldots
\geq (1-\eta)p\,.
\end{equation}
In particular, $\tilde{p}_{-q_-}(s,x,t,y)\geq 0$ provided $s<t<s+h$.
By the 
Chapman-Kolmogorov equation, $\tilde{p}_{-q_-}(s,x,t,y)$ is nonnegative for
all times $s,t$. Also,
$$
\tilde{p}_{-q_-}
=p-\left(P^{q_-}p-(P^{q_-})^2 p\right)-\left((P^{q_-})^3p-(P^{q_-})^4 p\right)
+\ldots
\leq p\,,
$$
provided $s<t<s+h$. Therefore,
\begin{equation}
  \label{eq:npbbp}
\tilde{p}_{-q_-}(s,x,t,y)\leq p(s,x,t,y)\,,\quad \mbox{ for all } s,t\in
\RR\,,\; x,y\in X\,. 
\end{equation}
This further yields that the perturbation of $\tilde{p}_{-q_-}$ by $q_+$ is well
defined, and $\widetilde{(\tilde{p}_{-q_-})}_{q_+}\leq \tilde{p}_{q_+}$, compare (\ref{eq:18.75}).
On the other hand the (nonnegative) series defining $\tilde{p}_{|q|}$ 
is  convergent.  Thus, by
Lemma~\ref{lem:aep} and the considerations above, 
$$
\tilde{p}_{-q_-}\leq \widetilde{(\tilde{p}_{-q_-})}_{q_+}=\tilde{p}_q \leq \tilde{p}_{q_+}\,.
$$
We note that if $n$ is a natural number and 
$s<t<s+nh$, then Chapman-Kolmogorov, (\ref{eq:lbbp}),
Theorem~\ref{th:optyl}, and (\ref{eq:ffb}) yield that, for all $x,y\in X$,
\begin{equation}
  \label{eq:comp}
(1-\eta)^n \leq \frac{\tilde{p}_q(s,x,t,y)}{p(s,x,t,y)}\leq
\frac{1}{1-\eta}\exp\frac{n\eta }{1-\eta}\,.
\end{equation}
By Theorem~\ref{th:optyl} (applied to $|q|$), the  series
$\sum_{n=0}^\infty (P^q)^n p$ is  absolutely convergent
and hence
$$
(I-P^q)\tilde{p}_q=\sum_{n=0}^\infty (P^q)^n p-\sum_{n=0}^\infty (P^q)^{n+1} p=p\,.
$$
Thus $\tilde{p}=\tilde{p}_q$ solves (\ref{eq:df}).

To prove the uniqueness of the solution, let $\tilde{p}$ 
be any transition density which locally in time 
is comparable with $p$. Then the 
integral in (\ref{eq:df}) is absolutely convergent. 
This follows from Lemma~\ref{lem:intercept} (note that the domain of
integration in (\ref{eq:df}) is merely $(s,t)\times X$).
Therefore $(I-P^q)\tilde{p}=p$ and
$$
\tilde{p}=\sum_{n=0}^\infty (P^q)^n (I-P^q)\tilde{p}
=\sum_{n=0}^\infty (P^q)^np=\tilde{p}_q\,.
$$
If $\eta^*(|q|)=0$ then we can have $0<\eta<1$ arbitrarily
small in the above discussion, therefore $\tilde{p}_q$ and $p$ are
asymptotically equal  by (\ref{eq:comp}).
\end{proof}

Let $\tilde P=P_{\tilde p}$, that is $\tilde{P}f(s,x)=\int_\RR\int_X
\tilde{p}(s,x,u,z)f(u,z)\,dzdu$, where $\tilde p=\tilde{p}_q$. We like
to note that,
\begin{equation}
  \label{eq:dPt}
  \tilde{P}=\sum_{n=0}^\infty (P^q)^nP\,.
\end{equation}

Following \cite{MR1488344}, \cite{MR2395164}, 
we will say that $q\,: \RR \times  X\to \RR$ belongs to the
parabolic (space-time) {\em Kato class} for $p$, 
if  
\begin{equation}
  \label{eq:Kc1}
\lim_{h \to 0+} \sup_{s\in \RR\,,\;x \in  X} 
\int_ X 
\int_s^{s+h} 
p(s,x,u,z)
|q(u,z)|
\,dudz = 0\,,
\end{equation}
and
\begin{equation}  \label{eq:Kc2}
\lim_{h \to 0+} \sup_{t\in \RR\,,\;y \in  X} 
\int_ X \int_{t-h}^t 
p(u,z,t,y)
|q(u,z)|\,dudz = 0\,.
\end{equation}
We say that $p$ is {\it probabilistic} if
\begin{equation}\label{eq:subprob}
\int_X p(s,x,t,y)dy= 1\,,\quad \mbox{ for } s<t\,,\; x\in X\,.
\end{equation}
\begin{lem}\label{l:icn}
If $p$ is probabilistic, $p(s,x,t,y)=p(s,y,t,x)$ for $s,t\in \RR$, $x,y\in X$, and $q$
is relatively Kato for $p$, then $q$ is in the parabolic Kato class.
\end{lem}
\begin{proof}
Let $h,\eta\geq 0$, and assume that for all $x,y\in \Rd$ and $s < t\leq s+h$,
$$
\int_s^t 
\int_{X}
p(s,x,u,z)p(u,z,t,y)|q(u,z)|\,dzdu
\le \eta \,p(s,x,t,y)\,.
$$
Integrating this with respect to $dy$ we obtain
$$
\int_s^t 
\int_{X}
p(s,x,u,z)|q(u,z)|\,dzdu
\le \eta\,.
$$
Thus, $\eta^*(|q|)=0$ yields (\ref{eq:Kc1}). Integrating with 
 respect to $dx$ instead, we obtain (\ref{eq:Kc2}).
\end{proof}
We note that Corollary~\ref{cor:equiv} below gives a special converse
to Lemma~\ref{l:icn}.

In view of Theorem~\ref{th:optyl}, the {\it relative Kato} condition
seems more intrinsic 
to Schr\"odinger
perturbations 
than the parabolic Kato condition 
(see \cite[Lemma 5.2]{MR2395164} and \cite[Theorem 3.10]{MR2253015}),
but the former may be more difficult to verify in specific situations.
For instance, the relative Kato condition for the Gaussian
kernel is rather difficult to explicitly characterize (as opposed to that for
the transition density of $\Delta^{\alpha/2}$, $0<\alpha<2$, see below).
It should be noted that each transition density $p$ determines a specific class of relatively
Kato functions $q$, and a detailed analysis is required to exhibit
interesting (unbounded) $q$.
We also remark that the relative Kato condition may be
interpreted as a Kato condition for {\it bridges}, see Section~\ref{sec:fd}.

\section{Examples}
\label{sec:ss}

In this section we assume that $X=\Rd$, $d\ge 1$, $\mathcal{M}$ is
 the $\sigma$-algebra of Borel subsets of $\Rd$, 
and $dz$ is the Lebesgue measure on $\Rd$. Let $\alpha \in (0,2)$. 
Our aim is to discuss the transition density, $\tilde p(s,x,t,y)$,
of $\Delta^{\alpha/2}+ q$, where $q$ is relatively small. 
To state our estimates of $\tilde{p}$, we specialize to
\begin{equation}
  \label{eq:pp}
p(s,x,t,y)=p_{t-s}(y-x)\,,
\end{equation}
where $x,y\in  \Rd$, $s<t$, and
$p_t$ is the smooth real-valued function on $\Rd$ determined by
\begin{equation}
  \label{eq:dpt}
  \int_ \Rd p_t(z)e^{iz\cdot\xi}\,dz=e^{-t|\xi|^\alpha}\,,\quad \xi\in
   \Rd\,, \quad t>0\,.
\end{equation}
In particular, for $\alpha=1$ we have
\begin{equation}
  \label{eq:tdCp}
p_t(z)=
\Gamma((d+1)/2)\pi^{-(d+1)/2}
\frac{t}{\big(|z|^2+t^2)^{(d+1)/2}}\,,  
\end{equation}
the Cauchy convolution semigroup.
Note that for every $\alpha\in (0,2)$,
\begin{equation}
  \label{eq:sca}
  p_t(z)=t^{-d/\alpha}p_1(t^{-1/ \alpha}z)\,,\quad t>0\,,\; z\in \Rd\,.
\end{equation}
This follows from (\ref{eq:dpt}).
We let $p(s,x,t,y)=0$ if $s\geq t$.
By the definitions, $p$~is time- and space-homogeneous: for all
$s,t,h\in\RR$, $x,y,z\in \Rd$ we have
$$p(s,x,t,y)=p(s+h,x+z,t+h,y+z)\,.$$
The semigroup $P_t f(x)=\int_ \Rd f(y) p_t(y-x) dy$
has $\Delta^{\alpha/2}$ as infinitesimal generator (\cite{BF}, \cite{Y},
\cite{MR1671973}, \cite{MR1873235}).  
Referring to Abstract, 
$p (s,x,t,y)$ is the fundamental solution of 
$\partial_t-\Delta_y^{\alpha/2}$ in the sense of distributions
(\cite{MR1671973}):
\begin{equation}
  \label{eq:fsol}
\int\limits_{\RR}\int\limits_{ \Rd}
p(s,x,t,y)\left[
\partial_t+\Delta^{\alpha/2}_y 
\right]\phi(t,y)\,dydt
 = -\phi(s,x)\,,
\end{equation}
where $s\in \RR$, $x\in  \Rd$, and $\phi\in C^\infty_c(\RR\times
\Rd)$. Here $C^\infty_c(\RR\times \Rd)$ is the class of all
infinitely differentiable compactly supported functions on $\RR\times
\Rd$, and 
\begin{eqnarray*}
   \Delta^{\alpha/2}\varphi(y) &=& \lim_{t \downarrow 0} \frac{ P_t\varphi(y) - \varphi(y)}{t} \\
   &=& 
\frac{2^{\alpha}\Gamma((d+\alpha)/2)}{\pi^{d/2}|\Gamma(-\alpha/2)|}
\lim_{\varepsilon \downarrow 0}\int_{\{|z|>\varepsilon\}}
   \frac{\varphi(y+z)-\varphi(y)}{|z|^{d+\alpha}}dz\,,\quad y\in \Rd\,.
   \end{eqnarray*}
A simple proof of (\ref{eq:fsol}) can be
given by using Fourier transform in the space variable, and (\ref{eq:dpt}) (we omit the details).

We will assume that $q$ is relatively small for $p$.
Let $\big(L\phi\big)(t,y)=\partial_t\phi(t,y)+\Delta^{\alpha/2}_y \phi(t,y)$.
We also introduce $\big(Q\phi\big)(t,y)=q(t,y)\phi(t,y)$, the operation of
multiplication by $q$.
Referring to our previous notation we have $P^q=PQ$, and 
(\ref{eq:dPt}) now reads
\begin{equation}
  \label{eq:dPtn}
  \tilde{P}=\sum_{n=0}^\infty (PQ)^nP\,.  
\end{equation}
We can interpret (\ref{eq:fsol}) as
\begin{equation}
  \label{eq:ifsol}
PL\phi=-\phi   \qquad (\phi\in C^\infty_c(\RR\times  \Rd))\,.
\end{equation}
This implies that 
\begin{equation}
  \label{eq:fsp}
  \tilde{P}(L+Q)\phi=-\phi \qquad (\phi\in C^\infty_c(\RR\times  \Rd))\,.
\end{equation}
Indeed, by (\ref{eq:ifsol}), 
\begin{eqnarray*}
\tilde{P}(L+Q)\phi&=&
\sum_{n=0}^\infty P(QP)^n(L+Q)\phi\\
&=&PL\phi+\sum_{n=1}^\infty (PQ)^nPL\phi+
\sum_{n=0}^\infty (PQ)^{n+1}\phi=-\phi\,.
\end{eqnarray*}
The associativity of operations involved, which we have used freely above, follows from Fubini's
theorem. Indeed, each $\phi(\cdot,\cdot)\in C^\infty_c(\RR\times  \Rd)$ is bounded by a constant
multiple of $p(\cdot,\cdot,t_0,y_0)$ for some $t_0\in \RR$, $y_0\in \Rd$, and our
remarks from the proof of Theorem~\ref{Theorem1g} apply.
This proves (\ref{eq:fsp}):
$$
\int\limits_{\RR}\int\limits_{ \Rd}
\tilde{p}(s,x,t,y)\big[
\partial_t\phi(t,y)+\Delta^{\alpha/2}_y \phi(t,y)+q(t,y)\phi(t,y)\big]\,dydt
 = -\phi(s,x)\,,
$$
where $s\in \RR$, $x\in  \Rd$, and $\phi\in C^\infty_c(\RR\times  \Rd)$.
In fact, $(s,x,t,y)\mapsto\tilde{p}(s,x,t,y)$ is continuous, except
when $s=t$. Indeed, continuity is first proved inductively for each $p_n$ in
(\ref{eq:defk}), by using an argument of uniform integrability. 
We omit the details of the proof, and refer the reader to
a similar argument in the proof of \cite[Lemma 14]{MR2283957}
(see also Lemma~\ref{l:icn} and (\ref{eq:oopn}) above).
The continuity of $\tilde{p}=\tilde{p}_q$ then follows from the locally (in time)
uniform convergence of the series in (\ref{eq:lcp}). 
\begin{proof}[Proof of Theorem~\ref{Theorem1}]
We consider the Gaussian transition density $g$ introduced 
in Section~\ref{sec:intro}. This corresponds to $\alpha=2$ in
(\ref{eq:pp}), and an analogue of (\ref{eq:fsol}) holds for $g$ and
the Laplacian with a similar proof. 
The above discussion of the fractional Laplacian applies also to the
Laplacian, provided $q$ is relatively small with respect to $g$.
\end{proof}
Apart from obvious similarities, there exist important differences between $p$ ($0<\alpha<2$)
and $g$ ($\alpha=2$). For instance the global decay of $p$
in space is qualitatively different from that of $g$. In fact we have the following
estimate of $p$ (cf. (\ref{eq:tdCp}) and see, e.g., \cite{BSS} for a proof).
\begin{lem}\label{l:ptxy}
There exists $c=c(d,\alpha)$ such that, for all $z\in  \Rd$, $t>0$,
$$
        c^{-1} \left(\frac{t}{|z|^{d+\alpha}} \land
        t^{-d/\alpha}\right) \le p_t(z) \le c
        \left(\frac{t}{|z|^{d+\alpha}} \land t^{-d/\alpha}\right)\,.
$$ 
\end{lem}
This {\it power-type} asymptotics yields the following 3P Theorem 
(\cite{MR2283957}).
\begin{thm} \label{3PLemma}
  There exists a constant $c=c(d,\alpha)$ such that 
$$
        p(s,x,u,z) \land p(u,z,t,y) \le c p(s,x,t,y) \qquad
        \mbox{ for }\quad
        x,z,y \in  X,\;\; s,u,t\in \RR\,.\label{3P:ineq}
$$ 
\end{thm}
For numbers $a,b\geq 0$ we have $ab = (a \vee b) (a \land b)$ and
$(a\vee b) \le a+b$. Therefore (\ref{3P:ineq}) yields
the following variant form with five occurrences of $p$:
\begin{equation}\label{5P:ineq}
p(s,x,u,z)p(u,z,t,y)\le c p(s,x,t,y) \big[p(s,x,u,z) + p(u,z,t,y)\big]\,.
\end{equation}
From this and Lemma~\ref{l:icn} we immediately obtain the following consequence.
\begin{cor}\label{cor:equiv}
For the transition density of the fractional
Laplacian $\Delta^{\alpha/2}$ with $0<\alpha<2$, the
parabolic Kato class equals the relative Kato class.
\end{cor}
For instance, if $\alpha<d$, then Lemma~\ref{l:ptxy} (see also (\ref{eq:sca})) yields
$$\int_0^{h} p_t(z)dt\approx |z|^{\alpha-d}\wedge
\big[h^2|z|^{-d-\alpha}\big]\,,\quad z\in \Rd\,,\; h>0\,.
$$
If $0<\varepsilon\leq \alpha$, $|q(u,z)|\leq
|z|^{-\alpha+\varepsilon}$, $s\in \RR$, $x\in \Rd$, and we let $h\to 0$, then
$$
\int_s^{s+h}  \int_\Rd  p(s,x,u,z) |q(u,z)| \,dzdu
\leq \int_\Rd |z|^{-\alpha+\varepsilon}\int_0^{h} p_t(z) \,dtdz \to 0\,.
$$
Therefore each such $q$ belongs to the parabolic Kato class, and so it is relatively
Kato. Let us note  that
the local (in time) comparability of the Schr\"odinger 
semigroups of the fractional Laplacian for $q$ in the
(parabolic) Kato class is a new result even in the autonomous case of \cite{MR1671973}.
We also refer the reader to \cite{MR1978999} for recent Gaussian results,
with a definition \cite[(1.2)]{MR1978999}, not unrelated to our
relative Kato condition (see also 
\cite[Definition 3.1]{MR1920109} and \cite[Lemma 5.2]{MR2395164}, \cite[Theorem 3.10]{MR2253015}).

For a study of other
consequences of the Kato condition for autonomous additive perturbations 
we refer the reader to \cite{MR1132313}, \cite{MR1825645}, \cite{MR2283957}.

\section{Further discussion}\label{sec:fd}
For the sake of clarity, let us add a comment on a lower bound 
in (\ref{eq:comp}). 
If $q$ is relatively Kato for $g$ (see Section~\ref{sec:intro}) then
for $\tilde{g}=\tilde{g}_q$ we  have
\begin{equation}
  \label{eq:optdK}
\frac{\tilde{g}_q(s,x,t,y)}{g(s,x,t,y)}\geq
\exp
\left(  
-\int_s^t\int_ X \frac{g(s,x,u,z)g(u,z,t,y)}{g(s,x,t,y)}q_- (u,z)\,dzdu
\right)\,.
\end{equation}
An analogous estimate holds for the transition 
density $p$ of the fractional Laplacian $\Delta^{\alpha/2}$.
These results follow from the fact that 
$$
R_\lambda:=(I+\lambda P^{q_-})^{-1}P^{q_-},\quad \lambda>0\,,
$$
is a sub-Markov resolvent (of $P^{q_-}$), a unique kernel satisfying
$$
  R_\lambda+\lambda P^{q_-} R_\lambda= P^{q_-}\,.
$$
For further background, 
 we refer the reader to \cite[7.2--7.7]{MR850715}.
By an argument of log-convexity (see, e.g., \cite{MR2207878} or \cite[p. 429]{MR1736524},
see also \cite[8.1-8.2]{MR850715}), we obtain (\ref{eq:optdK}). 
{Noteworthy, one only needs relative boundedness of $q_-$ to obtain
satisfactory lower bounds for $\tilde{p}$, as defined by (\ref{eq:df}).} 

We omit the
(standard) details for two reasons. Firstly, our emphasis in this
paper is on upper bounds, or non-explosion results. Secondly,
in view of possible generalizations mentioned below it seems
economical to postpone the full discussion to a later paper. 
We remark that, in principle, the lower bound (\ref{eq:optdK}) is well known 
(see, e.g., \cite{MR1329992}, \cite{MR1671973}).

We like to comment on possible and forthcoming generalizations of our results.
{
It is possible to extend the present results to more general integral
kernel or to {measures} (rather than
functions $q$, see \cite{MR2207878}, \cite{MR2319644}, \cite{MR2395164} for a related
study). In fact, considering $q(dudz)=\eta\delta_{u_0}(du)dz$, 
where $\eta\geq 1$ and $\delta_{u_0}$ is the probability measure 
concentrated in $u_0$, shows that $\tilde{p}_q$ may explode in finite time $u_0$.
}

The technique based on Theorem~\ref{th:optyl} applies to more general
{\it additive} perturbations (of the generator). 
{In studying these, one should attempt a natural and general description (in terms of $p$) of a
class of perturbations which lead to comparability
theorems.}
In this connection we refer to \cite{MR2316878} for a discussion of
{\it nonlocal perturbations} of the fractional Laplacian, and to \cite{MR2283957} for a
study of {\it gradient perturbations} of $\Delta^{\alpha/2}$ under the assumption that $1<\alpha<2$. 

{There is a deep well-known connection of Schr\"odinger operators
to the theory of multiplicative functionals of Markov processes, see,
e.g., \cite{MR2395164}. We
like to discuss the connection in the case of the Wiener process $Y$ in
$\Rd$, defined by the transition density $g$ of
Section~\ref{sec:intro} (note that $Y_t=B_{2t}$, where $B$ is the standard Brownian motion). 
Let $\EE_{s,x}$ and $\PP_{s,x}$ be respectively the expectation 
and the distribution of the process 
starting at the point $x\in \Rd$ at time $s\in \RR$, 
so that $\PP_{s,x}[Y_t\in A]=\int_Ag(s,x,t,y)\,dy$, 
where $Y_t$ is the canonical continuous coordinate process
evaluated at time $t>s$.
For $y\in \Rd$ we let $\EE_{s,x}^{t,y}$ and $\PP_{s,x}^{t,y}$ 
denote, respectively, the expectation and the distribution
of the process starting at $x$ at time $s$ and
conditioned to reach $y$ at time $t$  (Brownian bridge). 
The process is defined by the transition
probability density function
$r(u_1,z_1,u_2,z_2)=g(u_1,z_1,u_2,z_2)g(u_2,z_2,t,y)/g(u_1,z_1,t,y)$,
where $s\leq u_1<u_2<t$,
$z_1,z_2\in \Rd$. Thus, the finite dimensional distributions are given
by the density functions
\begin{equation}
  \label{eq:fdd}
\frac{g(s,x,u_1,z_1)g(u_1,z_1,u_2,z_2)\ldots g(u_n,z_n,t,y)}{g(s,x,t,y)}\,,
\end{equation}
and we have the following {\it disintegration} of $\PP_{s,x}$, 
\begin{eqnarray}
&&\PP_{s,x}\left(Y_{u_1}\in A_1\,,\ldots\,,Y_{u_n}\in A_n\,;\; Y_t\in B \right)\\\nonumber
&&= \int_B\PP_{s,x}^{t,y}\left(Y_{u_1}\in A_1\,,\ldots\,,Y_{u_n}\in
  A_n\right)g(s,x,t,y)\,dy\,.
\label{eq:disi} 
\end{eqnarray}
Here $x,z_1,\ldots,z_n, y\in \Rd$, $s\leq u_1<\ldots<u_n<t$, and $A_1, \ldots,
A_n\subset \Rd$ are Borelian.
Consider the {\it multiplicative functional} (\cite{MR1329992}) $e_q(s,t)=\exp(\int_s^t q(u,Y_u)\,du)$. We have
$$
\EE_{s,x}^{t,y} e_{q}(s,t)=\sum_{n=0}^\infty
\frac{1}{n!}\;\EE_{s,x}^{t,y}\;
\left(\int_s^t q(u,Y_u)\,du
\right)^n\,. 
$$
According to (\ref{eq:18.75}), and (\ref{eq:fdd}),  
\begin{eqnarray}
\EE_{s,x}^{t,y}
\;\int_s^t q(u,Y_u)\,du
&=&
\int_s^t\int_\Rd \frac{g(s,x,u,z) q(u,z)
  g(u,z,t,y)}{g(s,x,t,y)}\,dudz \label{eq:g1g}
\\
&=&\frac{g_1(s,x,t,y)}{g(s,x,t,y)}\,.\nonumber
\end{eqnarray}
Furthermore,
\begin{eqnarray*}
&&
\EE_{s,x}^{t,y}\;\frac{1}{2}\,
\left(\int_s^t q(u,Y_u)\,du\right)^2
= 
\EE_{s,x}^{t,y}\;
\int_s^t\int_u^t q(u,Y_u)q(v,Y_v)\,dvdu\\
&=&
\int_s^t\int_u^t \int_\Rd \int_\Rd 
\frac{g(s,x,u,z)g(u,z,v,w)g(v,w,t,y)}{g(s,x,t,y)}q(u,z)q(v,w)\,dwdz\,dvdu\\
&=&
\int_s^t\int_\Rd 
\frac{g(s,x,u,z)g_1(u,z,t,y)}{g(s,x,t,y)}q(u,z)\,dz\,du
=\frac{g_2(s,x,t,y)}{g(s,x,t,y)}\,.
\end{eqnarray*}
By induction, for every $n=0,1,\ldots$,
$$
\frac{1}{n!}\;
\EE_{s,x}^{t,y}\;
\left(\int_s^t q(u,Y_u)\,du
\right)^n =\frac{g_n(s,x,t,y)}{g(s,x,t,y)}\,,
$$
hence
\begin{equation}
  \label{eq:pqp}
\EE_{s,x}^{t,y} e_{q}(s,t)=\frac{\tilde{g}_q(s,x,t,y)}{g(s,x,t,y)}\,.
\end{equation}
We may interpret $\tilde{g}_q(s,x,t,y)/g(s,x,t,y)$ as the eventual inflation of mass of the
Brownian particle moving from $(s,x)$ to $(t,y)$. The mass grows
multiplicatively where $q>0$, and decreases where $q<0$. Thus we may
consider the results of the paper as uniform bounds for this mass.
 
The following example illustrates inequality (\ref{eq:optyl}).
For a general transition density $p$ we consider a function $q(u,z)=q(u)$, 
depending only on time and locally integrable in time. 
It easily follows from (\ref{eq:18.75}) 
that
$\tilde{p}_q(s,x,t,y)/p(s,x,t,y)=\exp\left(\int_s^tq(u)du\right)$. We like to note
that $\eta^*(|q|)=0$ in this example, whilst the emphasis in
Theorem~\ref{th:optyl} is on $\eta^*(|q|)>0$.

In view of (\ref{eq:g1g}) and (\ref{con:scp}),
the relative Kato class
may be considered as a Kato class for the {\it conditional} processes (bridges), see (\ref{eq:Kc1}, \ref{eq:Kc2}).
}

\noindent
{\bf Acknowledgment.} 
We thank Panki Kim, and the referees of the paper for
discussion, corrections and insightful suggestions. 
The results of this paper were presented at 2nd International Conference on
Stochastic Analysis and Its Applications, May 28-31, 2008 at Seoul
National University.
\bibliographystyle{abbrv}
\bibliography{nsp}

\def\polhk#1{\setbox0=\hbox{#1}{\ooalign{\hidewidth
  \lower1.5ex\hbox{`}\hidewidth\crcr\unhbox0}}}
  \def\polhk#1{\setbox0=\hbox{#1}{\ooalign{\hidewidth
  \lower1.5ex\hbox{`}\hidewidth\crcr\unhbox0}}}
\begin{thebibliography}{10}

\bibitem{BF}
C.~Berg and G.~Forst.
\newblock {\em Potential theory on locally compact abelian groups}.
\newblock Springer-Verlag, New York, 1975.

\bibitem{MR850715}
J.~Bliedtner and W.~Hansen.
\newblock {\em Potential theory}.
\newblock Universitext. Springer-Verlag, Berlin, 1986.
\newblock An analytic and probabilistic approach to balayage.

\bibitem{MR1671973}
K.~Bogdan and T.~Byczkowski.
\newblock Potential theory for the {$\alpha$}-stable {S}chr\"odinger operator
  on bounded {L}ipschitz domains.
\newblock {\em Studia Math.}, 133(1):53--92, 1999.

\bibitem{MR1825645}
K.~Bogdan and T.~Byczkowski.
\newblock Potential theory of {S}chr\"odinger operator based on fractional
  {L}aplacian.
\newblock {\em Probab. Math. Statist.}, 20(2, Acta Univ. Wratislav. No.
  2256):293--335, 2000.

\bibitem{MR2283957}
K.~Bogdan and T.~Jakubowski.
\newblock Estimates of heat kernel of fractional {L}aplacian perturbed by
  gradient operators.
\newblock {\em Comm. Math. Phys.}, 271(1):179--198, 2007.

\bibitem{BSS}
K.~Bogdan, A.~St{\'o}s, and P.~Sztonyk.
\newblock Harnack inequality for stable processes on {$d$}-sets.
\newblock {\em Studia Math.}, 158(2):163--198, 2003.

\bibitem{MR2137058}
K.~Bogdan and P.~Sztonyk.
\newblock Harnack's inequality for stable {L}\'evy processes.
\newblock {\em Potential Anal.}, 22(2):133--150, 2005.

\bibitem{MR2320691}
K.~Bogdan and P.~Sztonyk.
\newblock Estimates of the potential kernel and {H}arnack's inequality for the
  anisotropic fractional {L}aplacian.
\newblock {\em Studia Math.}, 181(2):101--123, 2007.

\bibitem{MR1473631}
Z.-Q. Chen and R.~Song.
\newblock Intrinsic ultracontractivity and conditional gauge for symmetric
  stable processes.
\newblock {\em J. Funct. Anal.}, 150(1):204--239, 1997.

\bibitem{MR1920109}
Z.-Q. Chen and R.~Song.
\newblock General gauge and conditional gauge theorems.
\newblock {\em Ann. Probab.}, 30(3):1313--1339, 2002.

\bibitem{MR1329992}
K.~L. Chung and Z.~X. Zhao.
\newblock {\em From {B}rownian motion to {S}chr\"odinger's equation}, volume
  312 of {\em Grundlehren der Mathematischen Wissenschaften [Fundamental
  Principles of Mathematical Sciences]}.
\newblock Springer-Verlag, Berlin, 1995.

\bibitem{MR936811}
M.~Cranston, E.~Fabes, and Z.~Zhao.
\newblock Conditional gauge and potential theory for the {S}chr\"odinger
  operator.
\newblock {\em Trans. Amer. Math. Soc.}, 307(1):171--194, 1988.

\bibitem{MR591851}
E.~B. Davies.
\newblock {\em One-parameter semigroups}, volume~15 of {\em London Mathematical
  Society Monographs}.
\newblock Academic Press Inc. [Harcourt Brace Jovanovich Publishers], London,
  1980.

\bibitem{MR0117523}
N.~Dunford and J.~T. Schwartz.
\newblock {\em Linear {O}perators. {I}. {G}eneral {T}heory}.
\newblock With the assistance of W. G. Bade and R. G. Bartle. Pure and Applied
  Mathematics, Vol. 7. Interscience Publishers, Inc., New York, 1958.

\bibitem{MR1397498}
R.~L. Graham, D.~E. Knuth, and O.~Patashnik.
\newblock {\em Concrete mathematics}.
\newblock Addison-Wesley Publishing Company, Reading, MA, second edition, 1994.
\newblock A foundation for computer science.

\bibitem{MR2395164}
A.~Gulisashvili.
\newblock Classes of time-dependent measures, non-homogeneous {M}arkov
  processes, and {F}eynman-{K}ac propagators.
\newblock {\em Trans. Amer. Math. Soc.}, 360(8):4063--4098, 2008.

\bibitem{MR2253111}
A.~Gulisashvili and J.~A. van Casteren.
\newblock {\em Non-autonomous {K}ato classes and {F}eynman-{K}ac propagators}.
\newblock World Scientific Publishing Co. Pte. Ltd., Hackensack, NJ, 2006.

\bibitem{MR1736524}
W.~Hansen.
\newblock Harnack inequalities for {S}chr\"odinger operators.
\newblock {\em Ann. Scuola Norm. Sup. Pisa Cl. Sci. (4)}, 28(3):413--470, 1999.

\bibitem{MR2160104}
W.~Hansen.
\newblock Uniform boundary {H}arnack principle and generalized triangle
  property.
\newblock {\em J. Funct. Anal.}, 226(2):452--484, 2005.

\bibitem{MR2207878}
W.~Hansen.
\newblock Global comparison of perturbed {G}reen functions.
\newblock {\em Math. Ann.}, 334(3):643--678, 2006.

\bibitem{MR2195185}
W.~Hansen.
\newblock Simple counterexamples to the 3{G}-inequality.
\newblock {\em Expo. Math.}, 24(1):97--102, 2006.

\bibitem{MR1873235}
N.~Jacob.
\newblock {\em Pseudo differential operators and {M}arkov processes. {V}ol.
  {I}}.
\newblock Imperial College Press, London, 2001.
\newblock Fourier analysis and semigroups.

\bibitem{J3}
T.~Jakubowski.
\newblock The estimates for the {G}reen function in {L}ipschitz domains for the
  symmetric stable processes.
\newblock {\em Probab. Math. Statist.}, 22(2, Acta Univ. Wratislav. No.
  2470):419--441, 2002.

\bibitem{MR1335452}
T.~Kato.
\newblock {\em Perturbation theory for linear operators}.
\newblock Classics in Mathematics. Springer-Verlag, Berlin, 1995.
\newblock Reprint of the 1980 edition.

\bibitem{MR2316878}
P.~Kim and Y.-R. Lee.
\newblock Generalized 3{G} theorem and application to relativistic stable
  process on non-smooth open sets.
\newblock {\em J. Funct. Anal.}, 246(1):113--143, 2007.

\bibitem{MR2319644}
P.~Kim and R.~Song.
\newblock Estimates on {G}reen functions and {S}chr\"odinger-type equations for
  non-symmetric diffusions with measure-valued drifts.
\newblock {\em J. Math. Anal. Appl.}, 332(1):57--80, 2007.

\bibitem{MR2253015}
V.~Liskevich, H.~Vogt, and J.~Voigt.
\newblock Gaussian bounds for propagators perturbed by potentials.
\newblock {\em J. Funct. Anal.}, 238(1):245--277, 2006.

\bibitem{MR758799}
D.~Revuz.
\newblock {\em Markov chains}, volume~11 of {\em North-Holland Mathematical
  Library}.
\newblock North-Holland Publishing Co., Amsterdam, second edition, 1984.

\bibitem{Y}
K.~Yosida.
\newblock {\em Functional analysis}.
\newblock Classics in Mathematics. Springer-Verlag, Berlin, 1995.

\bibitem{MR1488344}
Q.~S. Zhang.
\newblock On a parabolic equation with a singular lower order term. {II}. {T}he
  {G}aussian bounds.
\newblock {\em Indiana Univ. Math. J.}, 46(3):989--1020, 1997.

\bibitem{MR1978999}
Q.~S. Zhang.
\newblock A sharp comparison result concerning {S}chr\"odinger heat kernels.
\newblock {\em Bull. London Math. Soc.}, 35(4):461--472, 2003.

\bibitem{MR1132313}
Z.~Zhao.
\newblock A probabilistic principle and generalized {S}chr\"odinger
  perturbation.
\newblock {\em J. Funct. Anal.}, 101(1):162--176, 1991.

\end{thebibliography}

\end{document}